\documentclass{birkjour}
 \usepackage{amsthm,amsfonts,amssymb,amsmath}

\renewcommand{\theequation}{\thesection.\arabic{equation}}

\numberwithin{equation}{section}

\def \ker {{\rm ker}\,}

\def \dim {{\rm dim}\,}

\newcommand{\uu}{\underline}

\newcommand{\m}{\mathcal}

\theoremstyle{plain}
\newtheorem{theorem}{Theorem}[section]
\newtheorem{corollary}[theorem]{Corollary}
\newtheorem{lemma}[theorem]{Lemma}
\newtheorem{proposition}[theorem]{Proposition}

\newtheorem{definition}[theorem]{Definition}

 \newtheorem{example}[theorem]{Example}

\begin{document}

%
%
%
%
%
%
%
%
%

\title[Functional Models and
 Minimal Contractive Liftings]
 {Functional Models and
 Minimal Contractive Liftings}

\author[Santanu Dey]{Santanu Dey}
\address{%
Department of Mathematics, Indian Institute of Technology Bombay, \\
Powai,
Mumbai-400076, India
}
\email{dey@math.iitb.ac.in}

\author{Rolf Gohm}
\address{
Department of Mathematics and Physics, Aberystwyth University, \\
Aberystwyth SY23 3BZ, United Kingdom
}
\email{rog@aber.ac.uk}

\author[Kalpesh J. Haria]{Kalpesh J. Haria}
\address{%
Department of Mathematics, Indian Institute of Technology Bombay, \\
Powai,
Mumbai-400076, India
}
\email{kalpesh@math.iitb.ac.in}



\date{June 18, 2014}

\begin{abstract}
Based on a careful analysis of functional models for contractive multi-analytic operators we establish a one-to-one correspondence between unitary equivalence
classes of minimal contractive liftings of a row contraction and injective symbols of contractive multi-analytic operators. This allows an effective
construction and classification of all such liftings with given defects. Popescu's theory of characteristic functions of completely non-coisometric
row contractions is obtained as a special case satisfying a Szeg\"{o} condition. In another special case of single contractions and defects equal to
$1$ all non-zero Schur functions on the unit disk appear in the classification. It is also shown that the process of constructing liftings iteratively
reflects itself in a factorization of the corresponding symbols.
\end{abstract}

\maketitle

\section{Introduction}

Functional models from analytic functions were developed by Sz.Nagy-Foias \cite{NF70} and used for classifying contractive operators on Hilbert spaces. A similar approach
was used by Popescu in  \cite{Po89b} for classifying row contractions by certain multi-analytic operators. These classifying objects were called characteristic functions. In \cite{DG07} and more explicitly in \cite{DG11} characteristic
functions of liftings of row contractions were introduced and it was shown that they are complete invariants for unitary equivalence in a certain class of liftings.
Here we present an approach which is based on a systematic use of associated functional models which on the one hand exhibits Popescu's characteristic functions as special cases of characteristic functions of liftings and which on the other hand fully discloses the additional potentials of our generalization.
\\

\noindent
Let us immediately introduce the two main players.

(1) Let $\Gamma := \bigoplus^\infty_{n=0} (\mathbb C^d)^{\otimes n}$, the {\em full Fock space} over $\mathbb C^d$, and let $\{e_1, \ldots,e_d\}$ be the standard basis of ${\mathbb C}^d$.
Denote the left creation operator w.r.t. $e_j$ on $\Gamma$ by $L_j$, that is $L_j x = e_j \otimes x$ for $x \in (\mathbb C^d)^{\otimes n}$.  Let $\m D$ and $\m L$ be Hilbert spaces.
A linear operator $M_\Theta: \Gamma \otimes \m D \to \Gamma \otimes \m L$ is called {\em multi-analytic} (cf. \cite{Po95}) if it intertwines $L_j \otimes I$
for all $j=1,\ldots,d$. A multi-analytic operator $M_\Theta$ is determined by its symbol
$\Theta: \m D \to \Gamma \otimes \m L$
defined by $\Theta(\delta) := M_\Theta (e_\emptyset \otimes \delta)$ for all $\delta \in \m D$,
here $e_\emptyset$ denotes the standard basis vector of $(\mathbb C^d)^{0} = \mathbb C$.
We call $M_\Theta$ {\em contractive} if $\| M_\Theta \| \le 1$. Note that for $d=1$ and $\dim \m D = 1 = \dim \m L$ a contractive multi-analytic operator corresponds exactly to multiplication with a function in the Schur class, i.e., a bounded analytic function in the open unit disk with supremum norm at most $1$.
     
(2) A $d$-tuple $\uu C = (C_1, \ldots, C_d)$ of operators on a Hilbert space $\m H_C$ is called a {\em row contraction} if it is a contraction from
$\bigoplus^d_1 \m H_C$ to $\m H_C$ or, equivalently, if $\sum_j C_j C^*_j \le  I$.
If  a $d$-tuple $\uu E=(E_1,\ldots,E_d)$ on a Hilbert space $\m H_E = \m H_C \oplus \m H_A$ can be written in the form
 \[
 \uu E = \begin{pmatrix}
           \uu C   &  \uu 0 \\
 	     \uu B   &  \uu A
          \end{pmatrix}
 \]
for suitable $d$-tuples $\uu B$ and $\uu A$ then $\uu E$ is called a {\em lifting} of $\uu C$. The lifting is called {\em contractive} if $\uu E$ is still a row contraction and it is called {\em minimal} if $\m H_E$ is the smallest $\uu E$-invariant subspace containing $\m H_C$.
(By $\uu E$-invariance we mean invariance for all $E_j,\;j=1,\ldots,d$.)
We remark that it presents no particular difficulties to include sequences of operators ($d=\infty$) but we write all formulas as if $d$ is finite.

To establish a correspondence between the two main players we start with a detailed discussion of functional models. We use the generalized setting introduced by Popescu in \cite{Po89b} to study row contractions. The case $d=1$ (single contractions) is of course of special interest and our results are new also for $d=1$ but these results work just as well for general $d$.
This observation is important because from the case $d>1$ there are promising applications to
the dynamics of open quantum systems, see \cite{Go12} for an introduction to this topic and further references along these lines. The impact of our results on these applications will be worked out elsewhere.

In Section 2 we prove properties of functional models to be used later. The results about functional models which are new depend on observations about the geometry of an invariant subspace for a row isometry which we derive from a geometric lemma proved in an appendix to this paper. They establish relations between the positions of certain subspaces on the one hand and properties of the symbol $\Theta$ of the multi-analytic operator on the other hand.

Section 3 is the core of the paper.
Based on the results about the functional model and extending ideas from \cite{DG11}, we work out a mapping $\m E$ which, for any given row contraction $\uu C$, maps contractive multi-analytic operators $M_\Theta: \Gamma \otimes \m D \to \Gamma \otimes \m D_C$, where $\m D_C$ is the defect space of $\uu C$, to contractive liftings $\uu E$ of $\uu C$.
In the converse direction we make use of the theory of the minimal isometric dilation for the contractive lifting $\uu E$ to construct a mapping $\m M$ from contractive liftings $\uu E$  of a given row contraction $\uu C$ to
contractive multi-analytic operators $M_\Theta: \Gamma \otimes \m D \to \Gamma \otimes \m D_C$.
Suitably restricted, the maps $\m E$ and $\m M$ become inverses of each other and we obtain a one-to-one correspondence between unitary equivalence classes of minimal contractive liftings and equivalence classes of injective symbols. This justifies to call the corresponding multi-analytic operators characteristic functions of liftings, as has been done already in \cite{DG11}. With the theory developed here we provide a complete answer to the open problem posed at the end of Section 3 in \cite{DG11}, namely to classify the multi-analytic operators which can occur as characteristic functions of liftings. The surprisingly simple answer, all that is needed is the easily checked property of injectivity of the symbol, shows that the connection with liftings is a very natural application of multi-analytic operators and makes it now much easier to develop the applications.
We can always study minimal contractive liftings via the corresponding symbols. As an example of such an application we study the factorization of liftings with the help of the corresponding factorization of the characteristic functions 
(compare \cite{NF70}, \cite{Po06a}).

In Section 4 we revisit Popescu's work in \cite{Po89b} where he defined a characteristic function for a completely non-coisometric row contraction as a certain multi-analytic operator. We show that this can be considered as a special case of our theory in the sense that Popescu's characteristic functions appear as characteristic functions of a special type of liftings and that the property of being a complete invariant for unitary equivalence follows from the corresponding result about liftings in Section 3.

In Section 5 we look at a class of examples: minimal contractive liftings $E$ of a single contraction $C$ such that both $C$ and $E$ have defect equal to $1$. By our theory the unitary equivalence classes of these liftings are in one-to-one correspondence with non-zero Schur functions on the unit disk (up to unimodular complex factors). This gives us an opportunity to illustrate many of the previous results by easily computable examples. Already in this case in the future there is a lot more to find out about these liftings by making a more systematic use of what is known about Schur functions. For example, operator-valued Schur functions have been used for commutant lifting, see \cite{FF90} and more recently \cite{FHK06}.

Related work in different directions is done in  \cite{Ar98}, \cite{BV05}, \cite{DH14}, \cite{FF90}, \cite{FHK06}, \cite{FS12}, \cite{Go09} and \cite{Po06b}. But the explicit parametrization of all minimal contractive liftings achieved here is new and provides an excellent basis for studying applications, for example the dynamics of open quantum systems mentioned earlier. In fact in the special case of the minimal isometric dilation it is long known and well studied how it describes the embedding of open into closed quantum systems, see for example \cite{DG07,Go12}. Let us finish this introduction with a short reminder of the well known theory of the {\em minimal isometric dilation} $\uu V^T$ of a row contraction $\uu T = (T_1,\ldots,T_d)$ which appears in several places in this paper. It was first presented in \cite{Po89a}, with small modifications we use the notation from \cite{DG11}. In both papers a lot of additional details can be found.
Recall that a $d$-tuple $\uu V = (V_1, \ldots, V_d)$ on a Hilbert space $\m H$ is called a {\em row isometry} if $\uu V^* \uu V = I$ or, equivalently, the $V_j$ are isometries with orthogonal ranges.
For any row contraction $\uu T = (T_1,\ldots,T_d)$ on a Hilbert space $\m H_T$ there exists a row isometry $\uu V = (V_1,\ldots,V_d)$
on a bigger Hilbert space such that
$T_\alpha = P_{\m H_T} V_\alpha |_{\m H_T}$ for all $\alpha$. Here we use, as in similar cases, the notation $P_X$ for the orthogonal projection onto $X$ and the notation $T_\alpha$ for any word $\alpha = \alpha_1 \ldots \alpha_m$ with letters $\alpha_k \in \{1,\ldots,d\}$
to stand for the operator $T_{\alpha_1} \ldots T_{\alpha_m}$.
We call $m = |\alpha|$ the length of the word.
For all $\alpha$ means here and later: for all such words of all possible lengths including the empty word $\emptyset$ of length $0$ which corresponds to the identity operator.
A row isometry $\uu V$ with this property is called an isometric dilation of $\uu T$. It is easily seen that this is actually a very specific example of a contractive lifting in the sense introduced earlier. 

If we require minimality for an isometric dilation in the sense that the bigger Hilbert space is the smallest closed $\uu V$-invariant space containing $\m H_T$ then this determines $\uu V$ up to unitary equivalence and we denote it by $\uu V^T$, on the Hilbert space $\hat{\m H}_T \supset \m H_T$.
There is a construction of the minimal isometric dilation analogous to the Sch\"{a}ffer construction for a single contraction:
Recall that the operator $D_T: \bigoplus^d_{i=1} \m H_T \rightarrow
\bigoplus^d_{i=1} \m H_T$ given by
$D_T :=(\delta_{ij} I -T_i^*T_j)^\frac{1}{2}_{d\times d}$
is called the {\em defect operator},
$\m D_T:=\overline{\mbox{range~} D_T}$ is called the {\em defect space}
and the (Hilbert space) dimension of $\m D_T$ is called the {\em defect} of $\uu T$.
The minimal isometric dilation $\uu V^T$ is then realized in a canonical way on the Hilbert space $\hat{\m H}_T := \m H_T \oplus (\Gamma \otimes \m D_T)$. Again $\Gamma := \bigoplus^\infty_{n=0} (\mathbb C^d)^{\otimes n}$ is the full Fock space over $\mathbb C^d$ and we note that  $\uu V^T$ restricted to $\Gamma \otimes \m D_T$ is nothing but $\uu L \otimes I$, with $\uu L = (L_1,\ldots, L_d)$, the $d$-tuple of creation operators on $\Gamma$. We refer to the row contraction $\uu L \otimes I$ as the {\em canonical row shift} and to $e_\emptyset \otimes \m D_T$ as its {\em generating wandering subspace}. (As usual, the subspace is called wandering for $\uu L \otimes I$ because its translates $L_\alpha \otimes I \,(e_\emptyset \otimes \m D_T)$ over all words $\alpha$, including the empty word $\emptyset$, are orthogonal to each other.)
The minimal isometric dilation $\uu V^T$ is obtained from any isometric dilation $\uu V$ by restriction. The unitary equivalence to the canonical construction above is expressed by a canonical unitary from the defect space $\m D_T$ onto the $\uu V$-wandering subspace $\m L_T := \overline{\rm{span}} \{\m H_T, \uu V (\displaystyle \bigoplus_{1}^d \m H_T)\} \ominus \m H_T$ given by
$D_T (\bigoplus^d_{i=1} \xi_i) \mapsto
\sum^d_{i=1} (V_i-T_i) \xi_i$. 
It can be extended to a canonical unitary from
$\m H_T \oplus (\Gamma \otimes \m D_T)$ onto
$\m H_T \oplus \bigoplus_\alpha V_\alpha \m L_T$ (over all words $\alpha$, including the empty word $\emptyset$) intertwining the two versions of the minimal isometric dilation.
Compare \cite{NF70} and \cite{Po89a,Po89b}.

\section{Functional Models}

As already defined in the Introduction, let $\Gamma$ be the full Fock space over $\mathbb C^d$ and $\{e_1, \ldots,e_d\}$ be the standard basis of ${\mathbb C}^d$. Denote the left creation operator w.r.t. $e_j$ on $\Gamma$ by $L_j$. Let $M_\Theta: \Gamma \otimes \m D \to \Gamma \otimes \m L$ be a contractive multi-analytic operator where
 $\m D$ and $\m L$ are Hilbert spaces. We now define the model space
 $\m H_\Theta :=  (\Gamma \otimes \m L) \oplus \overline{\Delta_{\Theta}(\Gamma \otimes \m D)}$
 where $\Delta_{\Theta} := (I- M_\Theta^* M_\Theta)^\frac{1}{2}$ and an isometric map
 $W_\Theta: \Gamma \otimes \m D  \to  \m H_\Theta$ by
 \[
 W_\Theta \xi := M_\Theta \xi \oplus \Delta_{\Theta} \xi ~~{\rm for~ all} ~\xi \in \Gamma \otimes \m D.
 \]
 Since $ W_\Theta$ is isometric, its range
$\m H_{\m D} := \{M_\Theta \xi \oplus \Delta_{\Theta} \xi:  \xi \in \Gamma \otimes \m D\}$ is closed. Note that by this construction we can think of $M_\Theta$ as the orthogonal projection onto $(\Gamma \otimes \m L) \oplus 0$ restricted to $\m H_{\m D}$.

Further define the $d$-tuple $\uu V = (V_1, \ldots, V_d)$  on the Hilbert space
$\m H_\Theta =  (\Gamma \otimes \m L) \oplus \overline{\Delta_{\Theta}(\Gamma \otimes \m D)}$
by
\[
 V_j(\eta \oplus \Delta_{\Theta} \xi) := (L_j \otimes I) \eta \oplus   \Delta_{\Theta}(L_j \otimes I) \xi ~{\rm ~for~} \xi
 \in  \Gamma \otimes \m D,\, \eta \in \Gamma \otimes \m L.
\]
It is easy to see that $\uu V$ is a row isometry on $\m H_\Theta$.
Observe that
\[
 V_j ( M_\Theta \xi \oplus \Delta_{\Theta}\xi) =  (L_j \otimes I) M_\Theta \xi \oplus \Delta_{\Theta} (L_j \otimes I) \xi =
  M_\Theta (L_j \otimes I) \xi \oplus \Delta_{\Theta} (L_j \otimes I) \xi
 \]
for $\xi \in \Gamma \otimes \m D,  j =1, \ldots, d$. Thus $\m H_{\m D}$ is an invariant subspace for $\uu V$.

Define $V_j^{\m D} := V_j|_{\m H_{\m D}}$ for $ j =1, \ldots, d$ and
denote the row isometric $d$-tuple $(V_1^{\m D}, \ldots, V_d^{\m D})$ on $\m H_{\m D}$ by
$\uu V^{\m D}$.
Because for $\xi \in  \Gamma \otimes \m D$
\begin{eqnarray*}
 W_\Theta (L_j \otimes I) \xi &=&  M_\Theta (L_j \otimes I) \xi \oplus \Delta_{\Theta} (L_j \otimes I) \xi \\
 &=&  V_j ( M_\Theta \xi \oplus \Delta_{\Theta} \xi) =  V_j^{\m D} W_\Theta \xi
\end{eqnarray*}
for $ j =1, \ldots, d,$
we conclude that
 $ W_\Theta : \Gamma \otimes \m D \to \m H_{\m D}$ is a unitary operator
intertwining between the canonical row shift
on $\Gamma \otimes \m D$ and $\uu V^{\m D}$. Restriction to the generating wandering subspaces yields a unitary operator
\[
 W_\Theta|_{e_\emptyset \otimes \m D} :e_\emptyset \otimes \m D  \to \ker ({\uu V^{\m D}{\rm)}}^*.
\]
Further we define
\begin{eqnarray*}
\m H_A &:=& \m H_\Theta \ominus \m H_{\m D}  \\
\m L_A &:=& \overline{\rm{span}} \{\m H_A, \uu V (\displaystyle \bigoplus_{1}^d \m H_A)\} \ominus\m H_A,  \\
\m L_E &:=& \overline{\rm span}\{\m L_A, P_{\m H_{\m D}} (e_\emptyset \otimes \m L)\}.
\end{eqnarray*}
The use of the subscript $E$ will become clear in the context of the next section. The following geometric lemma helps to clarify the positions of these spaces.

\begin{lemma} \label{lem:geo}
\begin{itemize}
  \item [(a)]
$P_{\m H_{\m D}} \ker \uu V^* \subset \ker ({\uu V^{\m D}{\rm)}}^*$.
  \item [(b)]
$\ker ({\uu V^{\m D}{\rm)}}^*= \m L_{A} \oplus \big(\ker \uu V^*\cap\ \m H_{\m D} \big)$.
\end{itemize}
\end{lemma}

\begin{proof}
Because this follows from a general argument about adjoints we have postponed this argument to an appendix. In Lemma \ref{inv}
we choose $T \in \m B(\m H_1, \m H_2)$ to be the
row isometry $\uu V \in \m B( \bigoplus^d_{j=1} \m H_\Theta, \m H_\Theta)$ and we choose
$\m K_1 := \bigoplus^d_{j=1} \m H_{\m D}$ and
$\m K_2 := \m H_{\m D}$, so that $T \m K_1 \subset \m K_2$ because $\m H_{\m D}$ is an invariant subspace. 
Then $\tilde{T} = \uu V^{\m D},\; \m N_1 =
\bigoplus^d_{j=1} \m H_A,\; \m N_2 = \m H_A,\; \m L = \m L_A$ and we get (a) from Lemma \ref{inv}(ii) and (b) from Lemma \ref{inv}(iii).  
\end{proof}

Now from Lemma \ref{lem:geo}(b) we have
$\m L_A \subset \ker ({\uu V^{\m D}{\rm)}}^*$. Further $e_\emptyset \otimes \m L \subset \ker\uu V^*$
which implies $P_{\m H_{\m D}} (e_\emptyset \otimes \m L) \subset P_{\m H_{\m D}}  \ker\uu V^*  \subset \ker (\uu V^{\m D})^*$, the last inclusion is Lemma \ref{lem:geo}(a).
Hence also $\m L_E \subset \ker ({\uu V^{\m D}{\rm)}}^*$ and we conclude that
\[
\m L_A \subset \m L_E \subset \ker ({\uu V^{\m D}{\rm)}}^*.
\]
In the following we determine how equalities in these inclusions depend on properties of the symbol $\Theta$.
Recall that the symbol $\Theta$ of $M_\Theta$ is defined by
$\Theta (\delta) := M_\Theta (e_\emptyset \otimes \delta)$ for $\delta \in \m D$. We also write $\Delta_{\Theta} (\delta)$ as a shorthand for the formally correct $\Delta_{\Theta}(e_\emptyset \otimes \delta)$.
Since $W_{\Theta} (e_\emptyset \otimes \m D)=  \ker (\uu V^{\m D})^*$, we have
\[
 x \in \ker (\uu V^{\m D})^* \quad\Leftrightarrow\quad x = \Theta (\delta) \oplus  \Delta_{\Theta} (\delta)~{\rm for~some}~ \delta \in \m D.
\]

\begin{lemma} \label{lem:fm}
For $ x = \Theta (\delta) \oplus  \Delta_{\Theta} (\delta)\in \ker (\uu V^{\m D})^*$ for some $\delta\in \m D$ we have
\begin{itemize}
  \item [(a)] $x \perp \m L_A \quad\Leftrightarrow\quad \Theta (\delta) \in  e_\emptyset \otimes \m L$.
   \item [(b)] $x \perp \m L_E \quad\Leftrightarrow\quad \Theta (\delta) = 0$.
\end{itemize}
\end{lemma}

\begin{proof}(a)
From Lemma \ref{lem:geo}(b) it follows that for $x \in \ker (\uu V^{\m D})^*$ we have $x \perp \m L_A$ if and only if $x \in \ker \uu V^*$. This is the case if and only if  for all $y = \eta \oplus \Delta_{\Theta} \xi, \, \xi  \in \Gamma \otimes \m D, \eta \in\Gamma \otimes \m L$
and all $j=1,\ldots,d$
  \begin{eqnarray*}
 0&=&\langle V_j^*x, y\rangle = \langle  x,  V_j y \rangle\\
  &=& \langle  \Theta (\delta) \oplus  \Delta_{\Theta} (\delta),  (L_j \otimes I )\eta \oplus \Delta_{\Theta}  (L_j \otimes I )\xi \rangle\\
  &=& \langle  \Theta (\delta),  (L_j \otimes I )\eta \rangle +   \langle  e_\emptyset \otimes \delta,  \Delta_{\Theta}^2  (L_j \otimes I )\xi \rangle\\
    &=& \langle  \Theta (\delta),  (L_j \otimes I )\eta \rangle+   \langle  e_\emptyset \otimes \delta,
 (L_j \otimes I )\xi \rangle -  \langle  e_\emptyset \otimes \delta,
  M_\Theta^*M_ \Theta (L_j \otimes I )\xi \rangle\\
     &=& \langle  \Theta (\delta),  (L_j \otimes I )\eta \rangle
+ 0 -  \langle  M_ \Theta (e_\emptyset \otimes \delta),
M_ \Theta(L_j \otimes I ) \xi\rangle\\
    &=& \langle  \Theta (\delta),  (L_j \otimes I )(\eta -  M_ \Theta \xi)\rangle .
  \end{eqnarray*}
But $ \langle  \Theta (\delta),  (L_j \otimes I )(\eta -  M_ \Theta \xi) = 0$  for all $ \xi  \in
\Gamma \otimes \m D, \eta \in\Gamma \otimes \m L$ and all $j=1,\ldots,d$
if and only if $\Theta (\delta) \in  e_\emptyset \otimes \m L$,
and hence (a) follows.

\noindent
(b) $x \perp \m L_E$ if and only if $x \perp \m L_A$ and $x \perp P_{\m H_{\m D}} (e_\emptyset \otimes \m L).$
Because for $\ell \in \m L$
\begin{eqnarray*}
   \langle x, P_{\m H_{\m D}} (e_\emptyset \otimes \m \ell)\rangle &=&
 \langle P_{\m H_{\m D}}x, e_\emptyset \otimes \m \ell\rangle
 = \langle x, e_\emptyset \otimes \m \ell\rangle\\
 &=&\langle  \Theta (\delta) \oplus  \Delta_{\Theta} (\delta), e_\emptyset \otimes \m \ell \oplus 0 \rangle
= \langle  \Theta (\delta) , e_\emptyset \otimes \m \ell\rangle.
  \end{eqnarray*}
we infer that $x \perp  P_{\m H_{\m D}} (e_\emptyset \otimes \m L)$ if and only if
$\Theta(\delta) \perp e_\emptyset \otimes \m L.$  From this and part (a) we obtain
$x \perp \m L_E \Leftrightarrow \Theta (\delta) = 0$.
\end{proof}

\begin{definition} \label{non-constant}
We say that a symbol $\Theta: \m D \to \Gamma \otimes \m L$ has no constant directions
if $\Theta(\delta) \in e_\emptyset \otimes \m L$ only for $\delta=0$.
\end{definition}

The following result relating properties of certain subspaces in the functional model to properties of the symbol is crucial for our applications to characteristic functions later.

\begin{theorem} \label{fm}
Suppose $M_\Theta: \Gamma \otimes \m D \to \Gamma \otimes \m L$ is a contractive multi-analytic operator where
 $\m D$ and $\m L$ are Hilbert spaces.
Then the following statements
hold:
 \begin{itemize}
  \item [(a)]
     $  \m L_A  = \m L_E \quad\Leftrightarrow\quad [ \Theta(\delta) \in e_\emptyset \otimes \m L \;\Rightarrow\; \Theta(\delta)=0]$ for all $\delta \in \m D$.
 \item [(b)] $  \m L_E = \ker (\uu V^{\m D})^* \quad\Leftrightarrow\quad \Theta$ is injective.
 \item [(c)] $ \m L_A = \ker (\uu V^{\m D})^*
 \quad\Leftrightarrow\quad \Theta$ has no constant directions.
 \end{itemize}
 \end{theorem}

\begin{proof}
(a) is immediate from Lemma \ref{lem:fm}.

(b) Assume that $\ker (\uu V^{\m D})^* = \m L_E $. Let  $\delta \in \m D$ such that $\Theta (\delta) = 0$.
Then $x := \Theta (\delta) \oplus  \Delta_{\Theta} (\delta) \in \ker (\uu V^{\m D})^*$ but from Lemma \ref{lem:fm}(b) also $x  \perp \m L_E =\ker (\uu V^{\m D})^*.$
 So $x  = 0$ and this implies $\delta = 0$ because $W_\Theta|_{e_\emptyset \otimes \m D} :e_\emptyset \otimes \m D  \to \ker ({\uu V^{\m D}{\rm)}}^*$ is unitary. Thus $\Theta$ is an injective map.

Conversely, assume that $\Theta$ is injective.
Let $x \in \ker (\uu V^{\m D})^* \ominus \m L_E$. Then there exist $\delta \in \m D$ such that  $x = \Theta (\delta) \oplus  \Delta_{\Theta} (\delta)$ and $x \perp \m L_E$. From Lemma \ref{lem:fm}(b) we infer that $\Theta (\delta)= 0$.
Since $\Theta $ is injective, we obtain $\delta =0$ and $x =  \Theta (0) \oplus  \Delta_{\Theta} (0)= 0$.
Thus $\ker (\uu V^{\m D})^*=  \m L_E$.

Finally (c) follows by combining (a) and (b). Note that having no constant directions clearly implies injectivity.
\end{proof}

The traditional motivation for functional models is the study of the compression
$\uu A$ of $\uu V$ to $\m H_A$, that is $\uu A := (A_1, \ldots, A_d)$ with
$A_j := P_{\m H_A} V_j \,|_{\m H_A}$ for $j=1, \ldots, d$. In the following propositions we state some properties which will be used later.
Note that by construction $\uu V$ is an isometric dilation of $\uu A$
and the following criterion for minimality is our first application of
Theorem \ref{fm}. 

\begin{proposition} \label{min-dil}
$\uu V$ is a minimal isometric dilation of $\uu A$ if and only if $\m L_A = \ker (\uu V^{\m D})^*$ which happens if and only if $\Theta$ has no constant directions.
\end{proposition}

\begin{proof}
Restriction of $\uu V$ to the smallest closed $\uu V$-invariant subspace
containing $\m H_A$, with the wandering subspace $\m L_A$, provides a
minimal isometric dilation for $\uu A$.
Hence $\uu V$ itself is a minimal isometric dilation of $\uu A$ if and only if $\m L_A = \ker (\uu V^{\m D})^*$. The second half is a restatement of Theorem \ref{fm}(c).
\end{proof}

Recall that $\uu A$ is called {\em completely non-coisometric} if $z \in \m H_A$ satisfies $\sum_{|\alpha|=n} \| A^*_\alpha z \|^2 = \|z\|^2$ for all $n \in \mathbb{N}$ only if $z=0$.

\begin{proposition} \label{model}
The following statements hold:
\begin{itemize}
\item [(a)]  $\m H_A \cap (\Gamma \otimes \m L)^\perp = \{0\}$.
\item [(b)] $\uu A$ is completely non-coisometric.
\end{itemize}
\end{proposition}

\begin{proof}
If $z \in \m H_\Theta$ is orthogonal to $\Gamma \otimes \m L$ then it can be written
in the form $0 \oplus \Delta_\Theta \eta$ with $\eta \in \Gamma \otimes \m D$. If further
$z \in \m H_A$ then it is also orthogonal to $M_\Theta \xi \oplus \Delta_\Theta \xi$ for all $\xi \in \Gamma \otimes \m D$. With $\xi = \eta$ we find $\Delta_\Theta \eta = 0$ and hence $z=0$. This proves (a).

If $z \in \m H_A$ satisfies $\sum_{|\alpha|=n} \| A^*_\alpha z \|^2 = \|z\|^2$ for all $n \in \mathbb{N}$ then $z$ must be orthogonal to
$\Gamma \otimes \m L$ because $A^*_\alpha = V^*_\alpha  \,|_{\m H_A}$ and $\uu V$ acts on $\Gamma \otimes \m L$ as a row shift.
Hence (b) follows from (a).
\end{proof}

Note that (b) has already been observed in \cite{Po89b}, in the proof of Theorem 5.1 there. We discuss the characteristic function of $\uu A$ in Section 4 and then compare our concept of having no constant directions with the traditional concept of a purely contractive function.

\section{From Contractive Multi-Analytic Operators to Contractive Liftings and Back}

In this section we state and prove the main result of this paper, Theorem \ref{char}, which gives a complete classification of minimal contractive liftings by multi-analytic operators (or their symbols).  
The correspondence is constructive and can be used to analyze the structure of the set of minimal contractive liftings. To be able to state this result in an optimal way as a one-to-one correspondence we need to introduce suitable equivalence classes.
Recall that liftings and the concepts of contractivity and minimality have already been defined in the Introduction. We also use the notation which is given there. 

\begin{definition} \label{def:equivalent}
Let $\uu E$ and $\uu E'$ be liftings of a row contraction $\uu C$. If there exists a unitary $u: \m H_E \rightarrow \m H_{E'}$ which intertwines $\uu E$ and $\uu E'$ and restricts to the identity on the subspace $\m H_C$ corresponding to $\uu C$ then we say that the liftings $\uu E$ and $\uu E'$ are unitarily equivalent.

If $M$ and $M'$ are multi-analytic operators with symbols $\Theta: \m D \rightarrow \Gamma \otimes \m L$ and $\Theta': \m D' \rightarrow \Gamma \otimes \m L$ (with the same $\m L$) and
there exists a unitary $v: \m D \rightarrow \m D'$ such that
$\Theta' \circ v = \Theta$ then we say that $M$ and $M'$ (also $\Theta$ and $\Theta'$) are equivalent.
\end{definition}

Note that if $\Theta_1: \m D \rightarrow \Gamma \otimes \m L_1$ and $\Theta_2: \m D' \rightarrow \Gamma \otimes \m L_2$ are multi-analytic operators then even if $\dim \m L_1 = \dim \m L_2$ it only makes sense to say that they are equivalent after fixing an identification between the spaces $\m L_1$ and $\m L_2$, for example by specifying unitaries $u_1: \m L_1 \rightarrow \m L$ and $u_2: \m L_2 \rightarrow \m L$ to identify both $\m L_1$ and $\m L_2$ with a space $\m L$. In this sense 
the concept of equivalence is slightly different from the concept of coincidence used in \cite{NF70,Po89a,Po89b}. This is further discussed in Section 4. 
We are now able to state the main result.

\begin{theorem} \label{char}
Let  $\uu C$ be a row contraction on a Hilbert space $\m H_C$.
Then there is a one-to-one correspondence between unitary equivalence classes of minimal contractive liftings $\uu E$
of $\uu C$ and equivalence classes of injective symbols $\Theta: \m D \rightarrow \Gamma \otimes \m D_C$.
\\
The dimension of $\m D$ is equal to the defect of $\uu E$.
The correspondence is given explicitly by two mappings $\m E$ from
contractive multi-analytic operators to contractive liftings and $\m M$ in the opposite direction which are constructed below.
\end{theorem}

By this one-to-one correspondence it is justified, in the case of minimal contractive liftings $\uu E$, to call the contractive multi-analytic operator $\m M(\uu E)$ the {\em characteristic function} of the lifting $\uu E$. This terminology has been introduced in \cite{DG11}. With the theory developed here we have completely solved the open problem formulated at the end of Section 3 in \cite{DG11}, namely to classify the symbols which can occur as characteristic functions of minimal contractive liftings: exactly the injective symbols can and do occur. We give a precise statement as follows.

\begin{corollary} \label{char2}
The characteristic function of a minimal contractive lifting has always an injective symbol. Conversely, if $M$ is any contractive multi-analytic operator with an injective symbol then it is the characteristic function  of the minimal contractive lifting $\m E(M)$.
\end{corollary}

Let us also note here the following immediate consequence of Theorem \ref{char} which shows a surprising correspondence between liftings of different row contractions.

\begin{corollary} \label{C change}
Let $\uu C=(C_1,\ldots, C_d)$ and $\uu C'=(C'_1,\ldots, C'_d)$
be row contractions with the same defect $\ell$. Then there is a one-to-one correspondence between their sets of unitary equivalence classes of minimal contractive liftings.
\end{corollary}

\begin{proof}
This is obvious because both sets are in one-to-one correspondence with the classes of injective symbols described in Theorem \ref{char}. Note that to fix a specific correspondence you have to fix an identification between $\m D_C$ and $\m D_{C'}$, compare the remarks after Definition \ref{def:equivalent}.
\end{proof}

We now construct the mappings $\m E$ and $\m M$ and determine their properties in a sequence of propositions and lemmas, more and more closing in on a proof of Theorem \ref{char}. 
Let
$\uu C=(C_1,\ldots, C_d)$ be any row contraction on a Hilbert space $\m H_C$, fixed once and for all.
We first define the mapping $\m E$ from contractive multi-analytic operators $M_\Theta : \Gamma \otimes \m D \to \Gamma \otimes \m D_C$ to contractive liftings $\uu E$ of $\uu C$.
We introduce the Hilbert space
\begin{align*}
\hat{\m H} &:= \m H_C \oplus \m H_\Theta \\
&= \m H_C \oplus (\Gamma \otimes \m D_C) \oplus \overline{\Delta_{\Theta}(\Gamma \otimes \m D)} \\
&= \m H_C \oplus \m H_A \oplus \m H_D,
\end{align*}
where the notation comes from the functional model construction in Section 2, with $\m L = \m D_C$.
Note that on the subspace $\m H_\Theta$ we have the row isometry $\uu V$ introduced in Section 2. It restricts to the canonical row shift on $\Gamma \otimes \m D_C$ which is a reducing subspace for $\uu V$.
Now we also have the minimal isometric dilation $\uu V^C$ on the subspace $\m H_C \oplus (\Gamma \otimes \m D_C)$, in the canonical version stated in the final paragraph of the Introduction, which also restricts to the canonical row shift on $\Gamma \otimes \m D_C$. Hence there exists a row isometry on $\hat{\m H}$, which we call $\hat{\uu V}$, which restricts to $\uu V$ on $\m H_\Theta$ and restricts to
$\uu V^C$ on $\m H_C \oplus (\Gamma \otimes \m D_C)$. If we examine the decomposition 
$\hat{\m H} =
\m H_E \oplus \m H_D$ with $\m H_E := \m H_C \oplus \m H_A$
then we obtain on $\m H_E$ a contractive lifting of $\uu C$ as follows:
 \[
 \uu E = \begin{pmatrix}
           \uu C   &  \uu 0 \\
 	     \uu B   &  \uu A
          \end{pmatrix}
          \quad \text{defined by} \quad
 E_j := P_{\m H_E} \hat{V}_j |_{\m H_E}.\;j=1,\ldots,d.
 \]

We now define the mapping $\m E$ by setting $\m E(M_\Theta)$
equal to the contractive lifting $\uu E$ of $\uu C$ defined in this way.
We also write $\m E_C(M_\Theta)$ or $E_{C,\Theta}$ if we want to include the dependence on the original row contraction $\uu C$ in the notation.

Note further that $\hat{\uu V}$ is an isometric dilation of $\uu E$
and we also have $E_\alpha = P_{\m H_E} \hat{V}_\alpha |_{\m H_E}$ for all words $\alpha$.
The notation $\m L_E$ introduced in Section 2 is consistent with the notation for isometric dilations of $\uu E$ in the sense that
\[
\m L_E = \overline{\rm span}\{\m L_A, P_{\m H_{\m D}} (e_\emptyset \otimes \m D_C)\}
= \overline{\rm{span}} \{\m H_E, \hat{\uu V} (\displaystyle \bigoplus_{1}^d \m H_E)\} \ominus\m H_E,
\]
hence we find a version of the minimal isometric dilation $\uu V^E$ of
$\uu E$ by restricting $\hat{\uu V}$ to $\hat{\m H}_E :=
\m H_E \oplus \bigoplus_\alpha \hat{V}_\alpha \m L_E$. As usual
we identify this space (by a canonical unitary) with
$\m H_E \oplus (\Gamma \otimes \m D_E)$.
Now from Theorem \ref{fm}(b) we immediately conclude that under a rather weak assumption, injectivity of the symbol $\Theta$, our construction of $\hat{\uu V}$ already provides us with a minimal isometric dilation of the lifting $\uu E$:

\begin{proposition} \label{prop:injective}
For the lifting $\uu E = \m E(M_\Theta)$ of $\uu C$
\[
\hat{\uu V} = \uu V^E  \quad \Leftrightarrow \quad \m L_E = \ker (\uu V^{\m D})^* \quad \Leftrightarrow \quad \Theta \text{\;injective}
\]
\end{proposition}

This means that for an injective symbol $\Theta$ the minimal isometric dilation of the lifting $\uu E = E_{C,\Theta}$ can be built in a straightforward way from the minimal isometric dilation of $\uu C$ together with the functional model for $\Theta$.
\\

Now we construct the mapping $\m M$ in the opposite direction. We still have the row contraction $\uu C=(C_1,\ldots, C_d)$ on a Hilbert space $\m H_C$, fixed once and for all. Suppose that we are further given a contractive lifting
 \[
 \uu E
= \begin{pmatrix}
           \uu C   &  \uu 0 \\
 	     \uu B   &  \uu A
          \end{pmatrix}
\]
on a Hilbert space $\m H_E := \m H_C \oplus \m H_A$. We can now construct the minimal isometric dilation $\uu V^E$ on the Hilbert space $\m H_E \oplus (\Gamma \otimes \m D_E)$. Instead of
$e_{\emptyset} \otimes \m D_E$ we can also write $\m L_E$.
Clearly $\uu V^E$
is an isometric dilation of $\uu C$ on $\m H_C$ and hence we can find
a version of the minimal isometric dilation $\uu V^C$ by restricting
$\uu V^E$ to $\m H_C \oplus \bigoplus_\alpha V^E_\alpha \m L_C$, where
$\m L_C = \overline{\rm{span}} \{\m H_C, \uu V^E (\displaystyle \bigoplus_{1}^d \m H_C)\} \ominus\m H_C$ is a subspace with dimension equal to the defect of $\uu C$ and hence can be identified canonically with $\m D_C$. Both $\m L_E$ and $\m L_C$
are wandering subspaces of $\uu V^E$ and we can think of the orthogonal projection
onto $\bigoplus_\alpha V^E_\alpha \m L_C$
restricted to
$\bigoplus_\alpha V^E_\alpha \m L_E$
as a contractive multi-analytic operator
$M_\Theta: \Gamma \otimes \m D_E \rightarrow \Gamma \otimes \m D_C$.
That this is indeed a multi-analytic operator can be directly verified from
$\m L_C \subset \m H_E \oplus \m L_E$,
compare \cite{Go12}, Theorem 1.2,
for a systematic study of the construction of multi-analytic operators from pairs of wandering subspaces.

We set $\m M(\uu E) := M_\Theta$ which defines a map $\m M$
from contractive liftings $\uu E$  of a given row contraction $\uu C$ to
contractive multi-analytic operators $M_\Theta: \Gamma \otimes \m D_E \to \Gamma \otimes \m D_C$. We also write $\m M_C(\uu E)$ or $M_{C,E}$
if we want to include the dependence on the original row contraction $\uu C$ in the notation.

Recall from Section 2 that within the functional model we can in fact always think of $M_\Theta$ as such a restriction of an orthogonal projection. By inspection we observe that
if and only if in the construction of $\m E$ we end up with
$\hat{\m H} = \hat{\m H}_E$ the application of $\m M$ after $\m E$ just reconstructs the original $M_\Theta$. Thus
from Proposition \ref{prop:injective} we conclude:

\begin{proposition} \label{ME1}
\[
\m M \circ \m E (M_\Theta) = M_\Theta \quad \Leftrightarrow \quad \Theta \text{\;injective}
\]
\end{proposition}

To make the equality on the left explicit recall that in the functional model space $\m H_\Theta$ for $M_\Theta: \Gamma \otimes \m D \rightarrow \Gamma \otimes \m D_C$ with injective symbol $\Theta$
we have
\[
\m L_E = W_\Theta (e_\emptyset \otimes \m D),
\]
see Section 2 and Theorem \ref{fm}(b).
Now $\m L_E$ in the dilation space of $\uu E = \m E (M_\Theta)$, which is provided by the functional model space $\m H_\Theta$ together with $\m H_C$, is related by the canonical unitary to $\m D_E$. With this canonical identification of $\m D$ and $\m D_E$ we have the equality in Proposition \ref{ME1}. 
\\

It remains to determine and to examine the class of contractive liftings which correspond to multi-analytic operators in this way. It is here where the minimality of liftings comes into the game. We restate the definition of minimality in a more explicit way.

\begin{definition} \label{def:minimal}
A lifting $\uu E$ of $\uu C$ is called minimal if
\[
\overline{\rm{span}} \{ E_\alpha x \colon x \in \m H_C , \text{all words}\; \alpha\} = \m H_E \,.
\]
Here $\alpha  = \alpha_1 \ldots \alpha_m$ with $\alpha_k \in \{1,\ldots,d\}$ if $m\in {\mathbb N}$ and $\alpha = \emptyset$ if $m=0$.
\end{definition}

Equivalently, the lifting $\uu E$ is minimal if and only if $\m H_E$ is the smallest $\uu E$-invariant subspace containing $\m H_C$. Hence Definition \ref{def:minimal} is consistent with the definition of minimality given in the Introduction.

\begin{proposition} \label{prop:minimal}
For a contractive lifting
 \[
 \uu E
= \begin{pmatrix}
           \uu C   &  \uu 0 \\
 	     \uu B   &  \uu A
          \end{pmatrix}
\]
of a row contraction $\uu C$ the following are equivalent:
\begin{itemize}
\item [(a)] $\uu E$ is minimal.
\item [(b)]
$\m H_A \cap (\bigoplus_\alpha V^E_\alpha \m L_C)^\perp = \{0\}$.
\end{itemize}
\end{proposition}

\begin{proof}
Let $y \in \m H_E$. With $\perp$ denoting the orthogonal complement
in $\m H_E$ we have
\begin{eqnarray*}
& & y \in \big[ \overline{\rm{span}} \{ E_\alpha x \colon x \in \m H_C, \;\text{all words}\;\alpha\} \big]^\perp \\
&\Leftrightarrow&
y \perp E_\alpha x \;(= P_{\m H_E} V^E_\alpha x)
\;\text{for all}\; x \in \m H_C, \;\text{all words}\;\alpha   \\
&\Leftrightarrow&
y \perp P_{\m H_E} \big[ \m H_C \oplus (\bigoplus_\alpha V^E_\alpha \m L_C) \big]
= \m H_C + P_{\m H_E} \bigoplus_\alpha V^E_\alpha \m L_C \\
&\Leftrightarrow&
y \in \m H_A \cap (\bigoplus_\alpha V^E_\alpha \m L_C)^\perp
\end{eqnarray*}
We conclude that
$\big[ \overline{\rm{span}} \{ E_\alpha x \colon x \in \m H_C, \;\text{all words}\;\alpha\} \big]^\perp = \{0\}$ if and only if
$\m H_A \cap (\bigoplus_\alpha V^E_\alpha \m L_C)^\perp = \{0\}$, which implies the proposition.
\end{proof}

Remark: Comparing (b) with Lemma 3.5(iii) of \cite{DG11} shows that what we call a minimal lifting in this paper is the same as what was called a reduced lifting in \cite{DG11}. We prefer the terminology `minimal' because Definition \ref{def:minimal} above is simpler and because this is consistent with the terminology `minimal isometric dilation' which is the most important example of a minimal contractive lifting. A number of additional results about minimal or reduced liftings can be found in \cite{DG11},
we only include the following
basic property which foreshadows the connection between minimal liftings and functional models
established later.

\begin{proposition} \label{A cnc}
The right lower corner $\uu A$ of a minimal contractive lifting $\uu E$ of $\uu C$ is completely non-coisometric.
\end{proposition}

\begin{proof}
Taking adjoints in the definition of minimal lifting we see that $\uu E$ is minimal if and only if there is no $0 \not=x \in \m H_A = \m H_E \ominus \m H_C$ such that $E^*_\alpha x \in \m H_A$ for all $\alpha$. In other words, if $\uu E$ is minimal then for
all $0 \not=x \in \m H_A$ we find $\alpha$ such that
\[
\|A^*_\alpha x\| = \|P_{\m H_A} E^*_\alpha x\|
< \|E^*_\alpha x\|
\]
and because $\uu E$ is contractive it follows
that $\uu A$ is completely non-coisometric.

Alternatively we could use Corollary \ref{cor:minimal} below to establish that
$\uu A$ arises from a functional model and then quote Proposition \ref{model}(b) to get the result.
\end{proof}

\noindent
We now proceed to obtain minimal liftings from
multi-analytic operators.

\begin{proposition} \label{E minimal}
If $M: \Gamma \otimes \m D \rightarrow \Gamma \otimes \m D_C$ is any contractive multi-analytic operator then the lifting $\m E (M)$ of $\uu C$ is always minimal.
\end{proposition}

\begin{proof}
In the construction of $\uu E = \m E (M)$ using a functional model of $M$ it is the subspace $\Gamma \otimes \m D_C$ in the model which becomes the subspace $\bigoplus_\alpha V^E_\alpha \m L_C$ in the dilation space of $\uu E$. Hence by Proposition \ref{prop:minimal} minimality of $\uu E$ is satisfied if and only if in the model we have
$\m H_A \cap (\Gamma \otimes \m D_C)^\perp = \{0\}$.
But by Proposition \ref{model}(a) this is always the case.
\end{proof}

\begin{corollary} \label{minimal to injective1}
If $\Theta$ is an injective symbol then there exists a minimal lifting $\uu E$ of $\uu C$ such that $M_\Theta = \m M (\uu E)$.
\end{corollary}

\begin{proof}
Put $\uu E := \m E (M_\Theta)$. Then $\uu E$ is minimal by Proposition \ref{E minimal} and $\m M (\uu E) = \m M \circ \m E (M_\Theta) = M_\Theta$ by Proposition \ref{ME1}.
\end{proof}

\begin{proposition} \label{minimal to injective2}
If a contractive lifting $\uu E$ of $\uu C$ is minimal then the symbol of $\m M (\uu E)$ is injective.
\end{proposition}

\begin{proof}
Suppose the symbol of $\m M (\uu E)$ is not injective. By the definition of $\m M (\uu E)$ this means that there exists $0 \not= x \in \m L_E$ which is orthogonal to $\m H_C \oplus (\bigoplus_\alpha V^E_\alpha \m L_C)$. Further by definition $\m L_E =
\overline{\rm{span}} \{\m H_E, \uu V^E (\displaystyle \bigoplus_{1}^d \m H_E)\} \ominus\m H_E$. Hence there exists $j$ such that $y := (V^E_j)^* x \not= 0$. [In fact, assume that on the contrary $(V^E_j)^* x = 0$ for all $j=1,\ldots,d$. Then $x$ is orthogonal to $V^E_j \m H_E$ for all $j=1,\dots,d$. But $x \in \m L_E$, hence $x$ is also orthogonal to $\m H_E$. It follows that $x$ is orthogonal to $\m L_E$ and hence that $x=0$,
contradicting our assumption above.]
Note that $y \in \m H_E$ and
\[
\langle y,z \rangle = \langle x, V^E_j z \rangle = 0
\]
for all $z \in \m H_C \oplus (\bigoplus_\alpha V^E_\alpha \m L_C)$, by the assumption about $x$ and because the $V^E_j$ leave this subspace invariant.
Hence
$0 \not= y \in \m H_A \cap (\bigoplus_\alpha V^E_\alpha \m L_C)^\perp$ and
$\uu E$ is not minimal by Proposition \ref{prop:minimal}.
\end{proof}

But injectivity of the symbol of $\m M (\uu E)$ does not imply minimality of $\uu E$.
For example take a lifting $\uu E$ with $\uu B = 0$ so that $\uu E$ is a direct sum of $\uu C$ and $\uu A$. If $\m H_A \not= \{0\}$ this is clearly not minimal. In this case $\m L_E = \m L_C \oplus \m L_A$ and if $\m L_A = \{0\}$, that is if $\uu A$ is chosen to be isometric, then $\m M (\uu E)$ is the identity.

\begin{corollary} \label{ME2}
If $M$ is any contractive multi-analytic operator then
the symbol of $\m M \circ \m E (M)$ is injective.
\end{corollary}

\begin{proof}
This follows from Proposition \ref{E minimal} and Proposition \ref{minimal to injective2}.
\end{proof}

To obtain the one-to-one correspondence stated in Theorem \ref{char} 
we need to study the equivalence classes introduced in Definition \ref{def:equivalent}. 

\begin{proposition} \label{prop:equivalent}
Let $\uu E$ and $\uu E'$ be minimal contractive liftings of a row contraction $\uu C$. Then $\uu E$ and $\uu E'$ are unitarily equivalent liftings if and only if $M_{C,E}$ and $M_{C,E'}$ are equivalent.
\end{proposition}

\begin{proof}
The proof is easier to understand by always using the canonical identifications with full Fock spaces and defect spaces. Then $\m L_E = e_\emptyset \otimes \m D_E$ etc.

If $\uu E$ and $\uu E'$ are unitarily equivalent liftings then we can extend the unitary $u$ to a unitary $\hat{u}$ which intertwines the
minimal isometric dilations $\uu V^E$ and $\uu V^{E'}$ and restricts
to the identity on a space $\m H_C \oplus (\Gamma \otimes \m D_C)$
contained in both. By restricting $\hat{u}$ we obtain a unitary
$v: \m L_E \rightarrow \m L_{E'}$ providing the equivalence of
$M_{C,E}$ and $M_{C,E'}$. (Note that this direction is true even without assuming minimality.)

Conversely assume that $M_{C,E}$ and $M_{C,E'}$ are equivalent
via $v: \m D_E \rightarrow \m D_{E'}$.
Again we identify the space of the minimal isometric dilation of
$\uu C$ with $\m H_C \oplus (\Gamma \otimes \m D_C)$. Then by minimality of $\uu E$ we have $\m H_A \cap (\Gamma \otimes \m D_C)^\perp = \{0\}$ from Proposition \ref{prop:minimal}, with $\m H_A = \m H_E \ominus \m H_C$ and using the identifications announced in the beginning of the proof. Taking orthogonal complements in the space $\hat{\m H}_E$ of the minimal isometric dilation we note that $(\m H_A)^\perp = \m H_C \oplus (\Gamma \otimes \m D_E)$ and hence find
$\hat{\m H}_E = \overline{\rm{span}} \{\m H_C \oplus (\Gamma \otimes \m D_C), \Gamma \otimes \m D_E \}$. Similarly
$\hat{\m H}_{E'} = \overline{\rm{span}} \{\m H_C \oplus (\Gamma \otimes \m D_C), \Gamma \otimes \m D_{E'} \}$. Because $M_{C,E}$ resp. $M_{C,E'}$ are the orthogonal projections onto
$\m H_C \oplus (\Gamma \otimes \m D_C)$ restricted to
$\Gamma \otimes \m D_E$ resp. $\Gamma \otimes \m D_{E'}$ we can extend the unitary $I \otimes v: \Gamma \otimes \m D_E \rightarrow \Gamma \otimes \m D_{E'}$ to a unitary $\hat{u}: \hat{\m H}_E
\rightarrow \hat{\m H}_{E'}$ which is the identity on $\m H_C \oplus (\Gamma \otimes \m D_C)$ and intertwines the minimal isometric dilations $\uu V^E$ and $\uu V^{E'}$. Because $\m H_E = \hat{\m H}_E \ominus (\Gamma \otimes \m D_E)$ and $\m H_{E'} = \hat{\m H}_{E'} \ominus (\Gamma \otimes \m D_{E'})$ this restricts to a unitary
$u: \m H_E \rightarrow \m H_{E'}$ intertwining $\uu E$ and $\uu E'$ and being the identity on $\m H_C$.
\end{proof}

\begin{corollary} \label{cor:minimal}
For a contractive lifting $\uu E$ of a row contraction $\uu C$:
\[
\uu E \;\text{and}\;\, \m E \circ \m M (\uu E) \; \text{are unitarily equivalent}
\quad\Leftrightarrow\quad \uu E \;\text{minimal}.
\]
\end{corollary}

\begin{proof}
By Proposition \ref{E minimal} we know that $\m E \circ \m M (\uu E)$
is always minimal. Because minimality is preserved by unitary equivalence, $\uu E$ and $\m E \circ \m M (\uu E)$ can only be unitarily equivalent if $\uu E$ is minimal. Conversely, if $\uu E$ is minimal then
by Proposition \ref{minimal to injective2} the symbol of $\m M(\uu E)$
is injective, hence by Proposition \ref{ME1}
\[
\m M \circ \m E \circ \m M (\uu E) = \m M (\uu E)
\]
and now Proposition \ref{prop:equivalent} implies that $\uu E$ and $\m E \circ \m M (\uu E)$ are unitarily equivalent.
\end{proof}

We remark that there is no canonical identification of $\m H_E \ominus \m H_C$ and
$\m H_{\m E \circ \m M (\uu E)} \ominus \m H_C$ and this unitary
equivalence is the best we can expect.
\\

By combining the results obtained so far we are now able to complete the proof of Theorem \ref{char} and Corollary \ref{char2} as follows.

\begin{proof}
Let us denote by $\tilde{\m M}$ respectively $\tilde{\m E}$ the mappings between equivalence classes which are given by $\m M$ respectively $\m E$ on representatives. We have to prove that these are well defined and inverse to each other.

From Proposition \ref{prop:equivalent} we find that $\tilde{\m M}$ is well defined and it maps into classes of injective symbols by Proposition \ref{minimal to injective2}. Conversely, assume that $M$ and $M'$ have injective symbols which are equivalent to each other. The liftings
$\uu E = \m E(M)$ and $\uu E' = \m E(M')$ are both minimal by Proposition \ref{E minimal}.
By Proposition \ref{ME1} we have $\m M(\uu E) =
M$ and $\m M(\uu E') = M'$ and we conclude from
Proposition \ref{prop:equivalent} that $\uu E$ and $\uu E'$ are unitarily equivalent. Hence
$\tilde{\m E}$ is well defined and maps into classes of minimal liftings. We have $\tilde{\m M} \circ \tilde{\m E} = id$ from Proposition \ref{ME1} and we have $\tilde{\m E} \circ \tilde{\m M} = id$ by Corollary \ref{cor:minimal}. Proposition \ref{prop:injective}
shows that the dimension of $\m D$ is equal to the defect of $\uu E$.

Note further that Corollary \ref{char2} is nothing but a restatement of Proposition \ref{minimal to injective2} and Proposition \ref{ME1}
established above. 
\end{proof}

We have now completely proved Theorem \ref{char} and Corollary \ref{char2}. Let us have a look at two easy examples. First, if $\m H_A = \{0\}$ then we have the trivial lifting $\uu E = \uu C$. This is a minimal lifting and the characteristic function $M_{C,E}$ is the identity (here $\m L_E = \m L_C$). Second, if $\uu E$ is the minimal isometric dilation of the row contraction $\uu C$ then this is a minimal lifting and the characteristic function $M_{C,E}$ is the zero function. Here $\m L_E = \{0\}$, so it is the zero function on the zero space and hence injective: no contradiction to Theorem \ref{char}.
More complicated examples will appear in the following sections.

We finish this section with an application.
The one-to-one correspondence established in Theorem \ref{char} very naturally leads to the possibility to examine the structure of the set of minimal contractive liftings via the corresponding characteristic functions.
Along these lines we obtain a result about the factorization of the characteristic function of  a minimal lifting.

\begin{theorem} \label{factor}
Let $\uu C$ be a row contraction on a Hilbert space $\m H_C$. If $\uu E$ is
a minimal contractive lifting of $\uu C$ on a Hilbert space $\m H_E \supset \m H_C$ and $\uu E' = (E'_1, \ldots, E'_d)$ is
a minimal contractive lifting of $\uu E$ on a Hilbert space $\m H_{E'} \supset \m H_E$ then $\uu E'$ is a minimal contractive lifting of $\uu C$ and for the characteristic function we have
\[
M_{C,E'} = M_{C,E} M_{E,E'}.
\]
Conversely, if $\uu E'$ is
a minimal contractive lifting of $\uu C$ on a Hilbert space $\m H_{E'} \supset \m H_C$ and the characteristic function $M_{C,E'} :\Gamma
\otimes \m D_{E'} \to \Gamma \otimes \m D_C$
can be written as
\[
M_{C, E'} = M_1 M_2
\]
where $M_1: \Gamma
\otimes \m D \to \Gamma \otimes \m D_C$ and 
$M_2 : \Gamma \otimes \m D_{E'} \to \Gamma \otimes \m D$
are contractive multi-analytic operators, for a Hilbert space $\m D$ and both
with injective symbols,
then there exists a minimal contractive lifting
 $\uu E$ of $\uu C$ such that $\uu E'$
 is a minimal contractive lifting of $\uu E$, and
$M_1$ and $M_2$ are equivalent to $M_{C,E}$ and $M_{E, E'}$ respectively.
\end{theorem}

\begin{proof}
The first half is Theorem 4.1 of \cite{DG11}. Its proof can be simplified
in the present setting. It is easy to check that contractivity and minimality are both preserved if we iterate liftings. Further let
$\hat{\m H}_{E'} = \m H_{E'} \oplus (\Gamma \otimes \m D_{E'}),\,
\hat{\m H}_E = \m H_E \oplus (\Gamma \otimes \m D_E),\,
\hat{\m H}_C = \m H_C \oplus (\Gamma \otimes \m D_C)$ the spaces of minimal isometric dilations
for which we have $\hat{\m H}_{E'} \supset \hat{\m H}_E \supset
\hat{\m H}_C$ if we iterate liftings. We can think of $M_{C, E'}$ as
the orthogonal projection $P_{C,E'}$ from $\hat{\m H}_{E'}$ onto $\hat{\m H}_C$ restricted to $\Gamma \otimes \m D_{E'}$ because this maps to $\Gamma \otimes \m D_C$. Similar for $M_{C, E}$ and $M_{E, E'}$. With this observation the factorization $M_{C,E'} = M_{C,E} M_{E,E'}$ follows from the obvious factorization
$P_{C,E'} = P_{C,E} P_{E,E'}$.

Let us now prove the converse direction. We use the map $\m E$, in particular Corollary \ref{char2} and Proposition \ref{prop:injective}.
From $M_1$ with its injective symbol $\Theta_1: \m D \rightarrow \Gamma \otimes \m D_C$ we can build the minimal contractive lifting
$\tilde{\uu E} = E_{C,M_1}$ of $\uu C$ which has characteristic function $M_1$ and defect equal to $\dim \m D$. Then from $M_2$ with its injective symbol $\Theta_2: \m D_{E'} \rightarrow \Gamma \otimes \m D$ we can build the minimal contractive lifting $\tilde{E}' = E_{{\tilde E},M_2}$ of $\tilde{\uu E}$ which has characteristic function $M_2$.
We can think of $\tilde{\uu E}'$ as a minimal contractive lifting of $\uu C$ with characteristic function $M_1 M_2$, by the first part above (or Theorem 4.1 in \cite{DG11}).
Hence with the assumption that $\uu E'$ has characteristic function $M_1 M_2$ we can use Proposition \ref{prop:equivalent} to conclude that $\uu E'$ and
$\tilde{\uu E}'$ are unitarily equivalent as liftings of $\uu C$. If we use this unitary to rotate the lifting $\tilde{\uu E}$ then we obtain a minimal contractive lifting $\uu E$ of $\uu C$ with the properties required.
\end{proof}

\section{Characteristic Functions of Completely Non-coisometric Row Contractions}

Recall the notion of a characteristic function of a completely non-coisometric row contraction $\uu A$ on $\m H_A$ from \cite{Po89b}: The space $\hat{\m H}_A = \m H_A \oplus (\Gamma \otimes \m D_A)$ of the minimal isometric dilation $\uu V^A$ not only contains the wandering subspace $\m L_A = e_\emptyset \otimes \m D_A$ but also the wandering subspace
$\m L_{*,A} := \ker  (\uu V^A)^*$
arising in the Wold decomposition of $\uu V^A$ into the direct sum of a row unitary part and a row shift part. The space
$\m L_{*,A}$ is canonically identified with the $*$-defect space $\m D_{*,A}$, the closure of $D_{*,A} \m H_A$. Here
$D_{*,A} := (I - \uu A \uu A^*)^{\frac{1}{2}}$ and
the canonical unitary from $\m D_{*,A}$ onto $\m L_{*,A}$ is given by $D_{*,A} \xi \mapsto (I - \sum^d_{i=1} V^A_i A^*_i) \xi$ (where $\xi \in \m H_A$).
Then the orthogonal projection onto the space of the row shift part restricted to $\Gamma \otimes \m D_A$ can be thought of as a contractive multi-analytic operator $M_{\Theta_A}: \Gamma \otimes \m D_A \rightarrow \Gamma \otimes \m D_{*,A}$ with a symbol
$\Theta_A: \m D_A \rightarrow \Gamma \otimes \m D_{*,A}$.
This is called the {\em characteristic function} of $\uu A$.

It is verified in \cite{Po89b}, Theorem 4.1, that if we construct the functional model for $M_{\Theta_A}$ and then form the compression to $\m H_A$ (as we did more generally in Section 2) then we recover the original $\uu A$ we started from, up to unitary equivalence. Our Proposition \ref{model} implies that this is only possible if $\uu A$ is
completely non-coisometric and the functional model constructed from $M_{\Theta_A}$ adds the insight that all completely non-coisometric row contractions can be obtained in this way.

Recall that a contractive multi-analytic operator $M = M_\Theta$ with symbol $\Theta: \m D \rightarrow \Gamma \otimes \m L$ is called {\em purely contractive} if $\| P_{e_\emptyset \otimes \m L}\, \Theta(\delta) \| < \| \delta \|$ for all $0 \not= \delta \in \m D$.
We attach a name, motivated in Section 5, to another property which already appears in \cite{Po89b} as (5.1).

\begin{definition} \label{sz}
We say that a contractive multi-analytic operator $M: \Gamma \otimes \m D \to \Gamma \otimes \m L$ satisfies the Szeg\"{o} condition if
\[
\overline{\Delta(\Gamma \otimes \m D)} =
\overline{\Delta\big((\Gamma \otimes \m D) \ominus (e_\emptyset \otimes \m D)\big)},
\]
where $\Delta = (I- M^* M)^{\frac{1}{2}}$.
\end{definition}

Recall that multi-analytic operators $M$ and $M'$ with symbols $\Theta: \m D \rightarrow \Gamma \otimes \m L$ and $\Theta': \m D' \rightarrow \Gamma \otimes \m L'$ are said to {\it coincide} (compare
\cite{NF70,Po89b}) if there are
unitaries $v_{\m D}: \m D \rightarrow \m D'$ and $v_{\m L}: \m L \rightarrow \m L'$ so that $\Theta' \circ v_{\m D} = (I \otimes v_{\m L}) \circ \Theta$.

With this terminology we can now restate one of the main results obtained in sections 4 and 5 of \cite{Po89b},  as follows.

\begin{theorem} [\cite{Po89b}] \label{popescu}
 A contractive multi-analytic operator coincides with the characteristic function of a completely non-coisometric row contraction if and only if it is purely contractive and satisfies the Szeg\"{o} condition. 
\end{theorem}

We want to discuss the relationship between this result and our theory of characteristic functions of liftings. 

\begin{lemma} \label{pure}
Let $\Theta:\m D \to \Gamma \otimes \m L$ be a symbol. If $\Theta$ has no constant directions (see Definition \ref{non-constant}) then it is purely contractive. If the Szeg\"{o} condition for $M_\Theta$ holds then, conversely, purely contractive implies having no constant directions.
\end{lemma}

\begin{proof}
Recall that having no constant directions means that
$\Theta(\delta) \in e_\emptyset \otimes \m L$ only for $\delta=0$ (compare Definition \ref{non-constant}).
This implies purely contractive, in fact in the remaining case  $\Theta(\delta) \not\in e_\emptyset \otimes \m L$ it is obvious that
$ \| P_{e_\emptyset \otimes \m L} \Theta(\delta) \| < \|\delta\|$. The converse direction, going from purely contractive to having no constant directions, only fails if there exists $0 \not= \delta \in \m D$ such that $\Theta(\delta) \in e_\emptyset \otimes\m L$
but $\| \Theta(\delta) \| < \|\delta\|$. We show that this contradicts the
Szeg\"{o} condition. In fact, from $M (e_\emptyset \otimes \delta) = \Theta(\delta) \in e_\emptyset \otimes \m L$ we conclude that
$M (e_\emptyset \otimes \delta) \perp (L_j \otimes I) M (\Gamma \otimes \m D) = M (L_j \otimes I) (\Gamma \otimes \m D)$ for all $j=1,\ldots,d$, hence
$M^* M (e_\emptyset \otimes \delta) \perp (L_j \otimes I) (\Gamma \otimes \m D)$ for all $j$ and $M^* M (e_\emptyset \otimes \delta) \in e_\emptyset \otimes \m D$. We conclude that $\Delta^2 (\delta) = e_\emptyset \otimes \delta - M^* M (e_\emptyset \otimes \delta) \in e_\emptyset \otimes \m D$ and
\[
\Delta (\delta) \perp \Delta (L_j \otimes I) (\Gamma \otimes \m D)
\]
for all $j$. Because $ \| \Theta(\delta) \| < \|\delta\|$ we have
$\Delta (\delta) \not= 0$, so the Szeg\"{o} condition cannot hold.
\end{proof}

\begin{corollary} \label{pop-min}
The characteristic function of a completely non-coisometric row contraction has no constant directions and (hence) it is injective. The isometric dilation provided by the functional model is minimal.
\end{corollary}

\begin{proof}
From Theorem \ref{popescu} together with Lemma \ref{pure}
we conclude that we have no constant directions and then 
by Proposition \ref{min-dil} the isometric dilation provided by the functional model is minimal.
The second statement is also part of \cite{Po89b}, Theorem 5.1,
where an alternative proof for it can be found.
\end{proof}

\begin{theorem} \label{cnc}
Let $\uu C=(C_1,\ldots, C_d)$ be a row contraction on a Hilbert space $\m H_C$ and
 \[
 \uu E
= \begin{pmatrix}
           \uu C   &  \uu 0 \\
 	     \uu B   &  \uu A
          \end{pmatrix}
\]
be a minimal contractive lifting on a Hilbert space $\m H_E = \m H_C \oplus \m H_A$ with characteristic function $M_{C,E}: \Gamma \otimes \m D_E \rightarrow \Gamma \otimes \m D_C$. Then
$M_{C,E}$ is purely contractive and satisfies the Szeg\"{o} condition
if and only if $\m L_E = \m L_A$ and $\m L_C = \m L_{*,A}$. 
In this case, with the canonical identifications, the characteristic function $M_{C,E}: \Gamma \otimes \m D_E \rightarrow \Gamma \otimes \m D_C$ of the lifting $\uu E$ is equal to the characteristic function $M_{\Theta_A}: \Gamma \otimes \m D_A \rightarrow \Gamma \otimes \m D_{*,A}$ of the row contraction $\uu A$.

The characteristic function $M_{\Theta_A}$ of any completely non-coisometric row contraction $\uu A$ can be written in this form for any row contraction $\uu C$ with defect equal to $\dim \m D_{*,A}$.

From \cite{Po89b}: Two characteristic functions $M_{\Theta_A}$ and $M_{\Theta_{A'}}$ coincide if and only if $\uu A$ and $\uu A'$ are unitarily equivalent.
\end{theorem}

\begin{proof}
Note that because $\uu E$ is minimal the symbol $\Theta$ of $M_{C,E}$ is injective by Proposition \ref{minimal to injective2}.
Restricted to $\m H_\Theta = \hat{\m H}_E \ominus \m H_C$ the minimal isometric dilation $\uu V^E$ is given by the functional model row isometry $\uu V$ corresponding to $M_{C,E}$,
as described in Section 2.
With Theorem \ref{fm}(b) and Proposition \ref{min-dil} we see that here $\m L_A = \m L_E$ is equivalent to $\uu V$ on $\m H_\Theta$ being the minimal isometric dilation of $\uu A$. Recall that $\uu V$ can be represented by
\[
 V_j(x \oplus \Delta y) := (L_j \otimes I)x \oplus   \Delta (L_j \otimes I) y
\]
for $x \oplus \Delta y \in (\Gamma \otimes \m D_C) \oplus \Delta (\Gamma \otimes \m D_E)$.
The Szeg\"{o} condition says exactly that the ranges of the maps
$\Delta (L_j \otimes I)$ for all $j=1,\dots,d$  have as their closed linear span the whole space $\overline{\Delta (\Gamma \otimes \m D_E)}$, hence, under the condition $\m L_A = \m L_E$, it is equivalent to
$\m L_{*,A}\, (= \ker \uu V^*) = (e_\emptyset \otimes \m D_C) \oplus 0 = \m L_C$.

We can now verify the first part of the theorem. Note that if we have
$\m L_E = \m L_A$ and $\m L_C = \m L_{*,A}$ then $M_{C,E} = M_{\Theta_A}$ follows from the fact that both are described by the same restriction of an orthogonal projection.
Assume $\m L_E = \m L_A$. Then $\Theta$ has no constant directions by
Theorem \ref{fm}(b),(c)  and hence $M_{C,E}$ is also purely contractive by Lemma \ref{pure}.
Assume additionally $\m L_C = \m L_{*,A}$. Then the Szeg\"{o} condition holds, as shown above.

Conversely from the Szeg\"{o} condition together with $M_{C,E}$ purely contractive we conclude by Lemma \ref{pure} that $\Theta$ has no constant directions which implies $\m L_E = \m L_A$, again by Theorem \ref{fm}(b),(c). As shown above, under this condition the Szeg\"{o} condition also implies $\m L_C = \m L_{*,A}$. The first part of the theorem is proved.

If we start with any completely non-coisometric row contraction $\uu A$ then by Corollary \ref{pop-min} the functional model for $M_{\Theta_A}$ provides a minimal isometric dilation for $\uu A$. If we choose any $\uu C$ with
defect equal to $\dim \m D_{*,A}$ then we can build $\uu E = \m E (M_{\Theta_A})$
(as shown in Section 3) based on any (unitary) identification of $\m D_{*,A}$ and $\m D_C$. Then $\m L_{*,A} = \m L_C$.
Note that also $\m L_A = \m L_E$ from Proposition \ref{min-dil}.
We conclude that $M_{C,E} = M_{\Theta_A}$ by the first part. This shows that indeed
for any completely non-coisometric row contraction $\uu A$ the characteristic function $M_{\Theta_A}$ appears in suitable liftings by $\uu A$ as the characteristic function of the lifting.

If $\uu A$ and $\uu A'$ are unitarily equivalent then clearly $M_{\Theta_A}$ and $M_{\Theta_{A'}}$ coincide (because the minimal isometric dilations are unitarily equivalent). Conversely assume that $M_{\Theta_A}$ and $M_{\Theta_{A'}}$ coincide,
so we have unitaries 
$v: \m D_A \rightarrow \m D_{A'}$ and $v_*: \m D_{*,A} \rightarrow \m D_{*,A'}$ so that $\Theta_{A'} \circ v = (I \otimes v_*) \circ \Theta_A$. 
Now for any $\uu C$ with defect equal to $\dim \m D_{*,A} =  \dim \m D_{*,A'}$, we can choose any unitary $u': \m D_{*,A'} \rightarrow \m D_C$ and use it to identify the spaces $\m D_{*,A'}$ and $\m D_C$.
Then the unitary $u := u' \circ v_*: \m D_{*,A} \rightarrow \m D_C$
can be used to identify $\m D_{*,A}$ and $\m D_C$. Based on these identifications we construct $\uu E = \m E(M_{\Theta_A})$ and
${\uu E}' = \m E(M_{\Theta_{A'}})$ and we can check that, with these identifications, $M_{C,E} = M_{\Theta_A}$ and $M_{C,E'} = M_{\Theta_{A'}}$ are actually equivalent. By Proposition \ref{prop:equivalent} it follows that the corresponding liftings $\uu E$ and $\uu E'$ are unitarily equivalent and this implies, by restriction, the unitary equivalence of $\uu A$ and $\uu A'$.
\end{proof}

We finish this section with a few comments.

The final part of the proof of Theorem \ref{cnc} illustrates the relationship between the concepts of coincidence and equivalence,
compare the comments after Definition \ref{def:equivalent}.

Note that in the proof of Theorem \ref{cnc} the only thing we used from the theory of characteristic functions of completely non-coisometric row contractions (from \cite{Po89b}) was the fact that
the functional model constructed from the characteristic function provides a minimal isometric dilation. Given an independent proof of this fact we have obtained a new proof of Theorem \ref{popescu} because clearly Theorem \ref{popescu} follows from Theorem \ref{cnc}. We claim that the new context of liftings simplifies some of the original arguments. 

While in the theory of characteristic functions of completely non-coiso\-metric row contractions the concept of a purely contractive function is fundamental it seems that in the wider class of characteristic functions of liftings the concept of having no constant directions is more fundamental.

Theorem \ref{cnc} does not solve all natural questions about the relationship between the different notions of characteristic functions:
whenever we have a characteristic function of a lifting then we always have associated the characteristic function of the completely non-coisometric row contraction in the right lower corner. Only in the very special situation of Theorem \ref{cnc} the two functions are equal. 
There are some computations in the general case in \cite{DG11}, section 4, see also \cite{Go12}, section 2, but we leave a more systematic treatment of this topic to future work. 

We finally remark that in the original theory of Sz.-Nagy and Foias
\cite{NF70} for $d=1$ the more general case of a completely non-unitary contraction was treated and that Ball and Vinnikov in \cite{BV05} provided a generalization of it for $d>1$. The relationship of our setting with the theory in \cite{BV05} is another promising field of investigation.

\section{Examples from Schur Functions}

We are content in this section to illustrate our results by examples and do not develop a complete theory.
Recall that an analytic function $\Theta$ in the open unit disk which satisfies $\|f\|_{\infty} \leq 1$ is called a {\em Schur function.} We refer to \cite{Ho88} for further information about Schur functions.
We can think of $\Theta$ as a symbol of a contractive multiplication operator $M_\Theta$ on the {\em Hardy space} $H^2$, the Hilbert space of analytic functions in the open unit disk with square summable Taylor coefficients and isomorphic in this way to the unilateral sequence space $\ell^2$.
This is the special case $d=1$ and $\dim \m D = 1 = \dim \m L$ in our general scheme for contractive multi-analytic operators. Note that, to simplify notation, we use $\Theta$ in two different ways here: the Schur function $z \rightarrow \Theta(z)$ for $z \in \mathbb{C}$ with $|z|<1$ has a sequence of Taylor coefficients which is the image of $1 \in \mathbb{C} \simeq \m D$ under $\Theta$ if used in the way we introduced the symbol for multi-analytic operators. Note also in this respect that identifying a one-dimensional space with $\mathbb{C}$ is not canonical but involves the choice of a unimodular factor. We call the Schur function the {\em spectral representation} of $\Theta$.

Note further that a Schur function
$\Theta$ represents an injective symbol (in the sense used in previous sections) if and only if it is non-zero. Two Schur functions $\Theta$ and $\Theta'$ are equivalent in the sense of
Definition \ref{def:equivalent} if and only if
$\Theta' = c\,\Theta$ with $c \in \mathbb{C}$ and
$|c|=1$.

Let $C$ be a contraction with defect equal to $1$. Then by Theorem \ref{char} we have a one-to-one correspondence between unitary
equivalence classes of minimal contractive liftings $E$ (of $C$) with defect equal to $1$ and non-zero Schur functions (up to unimodular complex factors).
Applying Theorem \ref{fm} we confirm that for all
non-zero Schur functions $\dim \m L_E = 1$ (equal to the defect of $E$) and find in addition that always $\dim \m L_A \le \dim \m D = 1$ and that $\dim \m L_A=1$ if and only if
(the spectral representation of) $\Theta$ is not a constant function.
\\

\begin{example}
\rm
Let $ C = \displaystyle \frac{1}{2}$ be the contraction on a Hilbert space $\m H_C = \mathbb C$. The defect is equal to $1$.
 Consider the M\"{o}bius transformation
 $\Theta(z) = \displaystyle \frac{z - \alpha}{1-\bar{\alpha} z}, |\alpha| < 1.$
Because $\Theta$ is an inner function,  the multiplication operator
$M_{\Theta} :  H^2 \to H^2$ is an isometry and we note that
$\Delta : = (I - M_\Theta^* M_\Theta)^\frac{1}{2} = 0$.
The model space is $\m H_\Theta = H^2 \oplus 0 H^2 = H^2$.
In this case $M_\Theta H^2$ consists of all functions $f \in H^2$
with $f(\alpha)=0$. Hence
$\m H_A = H^2 \ominus M_\Theta H^2$ is one-dimensional and equal to the scalar multiples of the function $z \mapsto \displaystyle \frac{1}{1-\bar{\alpha}z} = \sum^\infty_{n=0} \bar{\alpha}^n z^n$.

We want to construct the lifting $E = \m E_C(M_\Theta) = E_{C,\Theta}$, compare Section 3 for the procedure and for the notation. We find $\hat{\m H} = \m H_C \oplus H^2$
with the isometry $\hat{V}$ so that $\hat{V} (1 \oplus 0) =  \displaystyle \frac{1}{2} \oplus \frac{\sqrt{3}}{2}$ (the second summand is a constant function in $H^2$) and the restriction of $\hat{V}$ to $H^2$ is the multiplication operator $M_z$ with the variable $z$. We compute
$B = P_{\m H_A} \hat{V} |_{\m H_C}
= \displaystyle \frac{\sqrt{3}}{2}(1- |\alpha|^2)^{\frac{1}{2}}$
and $A = P_{\m H_A} \hat{V} |_{\m H_A} = \alpha$ and conclude that the minimal contractive lifting of $ C = \displaystyle \frac{1}{2}$ with the M\"{o}bius transformation $\Theta(z) = \displaystyle \frac{z - \alpha}{1-\bar{\alpha} z}$ as its characteristic function is the (scalar) $2 \times 2$-matrix
 \[
E
= \begin{pmatrix}
            \displaystyle  \frac{1}{2}   &  0
           \vspace{0.1cm} \\
 	     \displaystyle  \frac{\sqrt{3}}{2}(1\!- \!|\alpha|^2)^{\frac{1}{2}}   &  \alpha
          \end{pmatrix}.
\]
It is easy to check that the defect of $E$ is indeed equal to $1$.
As a non-constant inner function the M\"{o}bius transformation $\Theta(z) = \displaystyle \frac{z - \alpha}{1-\bar{\alpha} z}$
clearly satisfies the conditions in Theorem \ref{cnc} and indeed it is well known that $\Theta$  is the characteristic function of the (completely non-coisometric) contraction $A = \alpha$.
 \end{example}

Let us see what Theorem \ref{cnc} tells us in general about Schur functions as characteristic functions. For this consider $H^2 \subset L^2(dt) $, the Hilbert space of square integrable functions on the unit circle with Lebesgue measure. This can be achieved by considering boundary values. For a Schur function $\Theta$ and $f \in H^2$ in general $\Delta_\Theta f = (I-M^*_\Theta M_\Theta)^{\frac{1}{2}} f
\not= (1- |\Theta|^2)^{\frac{1}{2}} f$ (the latter function is not always in $H^2$), but because
\[
\| (I-M^*_\Theta M_\Theta)^{\frac{1}{2}} f \|^2
= \|f\|^2 - \| \Theta f \|^2 = \| (1- |\Theta|^2)^{\frac{1}{2}} f \|^2
\]
we can replace $\Delta_\Theta$ and $\Delta_\Theta H^2$ by
the multiplication with $(1- |\Theta|^2)^{\frac{1}{2}}$ and the subspace $(1- |\Theta|^2)^{\frac{1}{2}} H^2 \subset L^2(dt) $
(using the unitary defined by $\Delta_\Theta f \mapsto (1- |\Theta|^2)^{\frac{1}{2}} f$).

\begin{proposition}  \label{cnc Schur}
Let $\Theta$ be a non-zero Schur function.
\begin{itemize}
\item[(a)] $M_\Theta$ is purely contractive if and only if (in the spectral representation) $|\Theta(0)| < 1$.
\item[(b)] $M_\Theta$ satisfies the Szeg\"{o} condition (see Definition \ref{sz}) if and only if (the spectral representation of) $\Theta$ satisfies the {\em spectral Szeg\"{o} condition}:
\[
\int^{2 \pi}_0 \log \big( 1 - |\Theta(e^{it})|^2 \big) \,dt = - \infty\,.
\]
\item[(c)] A Schur function $\Theta$ is the characteristic function of a completely non-coisometric contraction if and only if it satisfies
$|\Theta(0)| < 1$ and the spectral Szeg\"{o} condition.
\end{itemize}
\end{proposition}

\begin{proof}
(a) is obvious because $\Theta(0)$ is equal to the zero'th Taylor coefficient of $\Theta$. With the considerations preceding this proposition we see that the Szeg\"{o} condition (from Definition \ref{sz}) holds if and only if the constant polynomial $1$ can be approximated in norm by analytic polynomials $p$ with
$p(0)=0$ in the Hilbert space $L^2 \big( (1- |\Theta|^2)\,dt \big)$.
Here the non-negative function $ 1- |\Theta|^2$ plays the role of a density for Lebesgue measure. By a classical theorem of Szeg\"{o}
this is the case if and only if the spectral Szeg\"{o} condition holds,
see \cite{Ho88}, Chapter 4.
This proves (b). Now (c) is just the reformulation of Theorem \ref{cnc} in the setting of Schur functions.
\end{proof}

It is a classical result that a Schur function $\Theta$ satisfies what we have called the spectral Szeg\"{o} condition if and only if it is an extremal Schur function, see \cite{Ho88}, Chapter 9.
Generalizations of this result for the case of operator-valued Schur functions (still with $d=1$) are given in \cite{Tr89}, generalizations for $d>1$ are discussed in \cite{Po06b}, Section 1.4, in connection with the notion of prediction entropy.

Note at this point that any non-zero Schur function is the characteristic function of a lifting but only a very special subclass consists of  characteristic functions of completely non-coisometric contractions. We give an example outside of this special subclass.

\begin{example}
\rm
Again we start with the contraction $ C = \displaystyle \frac{1}{2}$
on a one-dimensional Hilbert space $\m H_C$ but this time we look at the Schur function $\Theta(z) = \displaystyle \frac{z}{2}$. Clearly this does not satisfy the spectral Szeg\"{o} condition. Again we want to construct
the minimal contractive lifting $E = E_{C,\Theta}$.

In this case the model space is $\m H_\Theta = H^2 \oplus  \displaystyle  \frac{\sqrt{3}}{2} H^2 = H^2 \oplus H^2$. Hence
$\hat{\m H} = \m H_C \oplus H^2 \oplus H^2$ and the isometry
$\hat{V}$ maps $1 \oplus 0 \oplus 0$ to $ \displaystyle  \frac{1}{2} \oplus \frac{\sqrt{3}}{2} \oplus 0$ and acts as $M_z \oplus M_z$ on
$H^2 \oplus H^2$. Further $g \oplus h \in \m H_A \subset \m H_\Theta = H^2 \oplus H^2$ if and only if
\[
\langle \frac{1}{2} M_z f \oplus \frac{\sqrt{3}}{2} f, g \oplus h \rangle = 0 \quad \text{for all} \; f \in H^2,
\]
which is satisfied if and only if $h = -  \displaystyle \frac{1}{\sqrt{3}} M^*_z g$.
So we have
\[
\m H_A = \{ g \oplus - \frac{1}{\sqrt{3}} M^*_z g \colon g \in H^2 \}.
\]
Here $B(c) = P_{\m H_A} \hat{V} (c) =  \displaystyle \frac{\sqrt{3}}{2} c \oplus 0$ for $c \in \m H_C$ (note that $M^*_z$ applied to a constant function yields zero)
and for $A = P_{\m H_A} \hat{V} |_{\m H_A}$ we find
\[
A (g \oplus - \frac{1}{\sqrt{3}} M^*_z g)
= P_{\m H_A} \big( M_z g \oplus - \frac{1}{\sqrt{3}} (g - g(0)) \big)
= M_z (g - \frac{g(0)}{4}) \oplus - \frac{1}{\sqrt{3}} (g - \frac{g(0)}{4}).
\]
The last equality can be checked by verifying that the difference
$M_z  \displaystyle \frac{g(0)}{4} \oplus -\frac{1}{\sqrt{3}}(-\frac{3}{4}g(0))$ is indeed orthogonal to $\m H_A$.
\end{example}

We note that any factorization $\Theta = \Theta_1 \Theta_2$ among non-zero Schur functions leads to a corresponding factorization of liftings, by Theorem \ref{factor}. In particular the lifting in the previous example can be factorized in many ways. But we postpone a detailed investigation of this phenomenon to another place.

We finally remark that the general case of minimal contractive liftings for a single contraction (the general case $d=1$) can be handled in a similar way. In fact, for $d=1$ a multi-analytic operator has a spectral representation by an operator-valued Schur function, i.e., a bounded analytic function $\Theta$ on the open unit disk such that $\Theta(z) \in \m B(\m D, \m L)$ and $\| \Theta(z) \| \le 1$ for all $z$ in the open unit disk. Then Theorem \ref{char} takes the form that there is a one-to-one correspondence between unitary equivalence classes of minimal contractive liftings $E$ with defect $\dim \m D$ (of a given contraction $C$) and operator-valued Schur functions with values in $\m B(\m D, \m D_C)$, up to a unitary on $\m D$ (the same for all $z$), and with the following injectivity property: if $0 \not= \delta \in \m D$ then $z \mapsto \Theta(z) \delta$ is not the zero function.

\section{Appendix}

In our analysis of functional models in Section 2 we need a few technical results about how the kernel of the adjoint changes if we go to a restriction of an operator. We provide them in the following lemma. These are quite general observations, useful in particular to describe the geometry of invariant subspaces.

\begin{lemma} \label{inv}
 Let $T \in \m  B (\displaystyle  \m H_1, \m H_2)$ where  $\m H_1, \m H_2$ are  Hilbert spaces and $\m K_i \subset \m H_i,\; i=1,2,$
 be subspaces such that $T \m K_1 \subset \m K_2$. Define $\tilde{T}:=T|_{\m K_1} \colon \m K_1 \mapsto \m K_2$ and  $\m N_i := \m H_i\ominus \m K_i,\; i=1,2$.
Then
\begin{itemize}
 \item [(i)] 
$\ker {\tilde T}^*= \{\xi \in \m K_2:  T^* \xi \in \m N_1 \}$.
\item [(ii)] 
$P_{\m K_2}\, \ker  T^* \subset \ker {\tilde T}^*$.
\end{itemize}

\noindent
Assume in addition that $T$ is an isometry. Set
$\m L := \overline{\rm{span}}\, \{\m N_2,\,  T \m N_1\} \ominus\m N_2$. Then

\begin{itemize}
\item[(iii)]
$\overline{\rm span}\,\{\m L, P_{\m K_2}\, \ker  T^*\} = \ker {\tilde T}^* = \m L \oplus(\ker  T^*\cap\ \m K_2)$.
\end{itemize}
\end{lemma}

\begin{proof}(i)
 For $\xi\in \m K_2$
 \begin{eqnarray*}
\xi\in \ker {\tilde T}^*
&\Leftrightarrow&  \langle \xi,  {\tilde T} \eta\rangle = 0
 ~{\rm for~ all }~ \eta \in   \m K_1\\
&\Leftrightarrow&  \langle \xi,  T\eta\rangle = 0
 ~{\rm for~ all }~ \eta \in  \m K_1\\
&\Leftrightarrow&  \langle  T^*\xi, \eta\rangle = 0
 ~{\rm for~ all }~ \eta \in  \m K_1\\
&\Leftrightarrow&  T^*\xi \in  \m N_1.
\end{eqnarray*}

\noindent(ii)  Decompose $\xi \in \ker  T^*$
as $\xi = \xi_{\m K_2} \oplus \xi_{\m N_2}$ such that $\xi_{\m K_2} \in \m K_2$ and $\xi_{\m N_2} \in \m N_2$.
Since $T^* (\m N_2) \subset \m N_1$ it follows that
$ T^*\xi_{\m N_2} \in   \m N_1$. Thus $T^*\xi_{\m K_2} =  T^* \xi -  T^*\xi_{\m N_2}
=0- T^*\xi_{\m N_2} \in  \m N_1$. We conclude that $\xi_{\m K_2} \in  \ker {\tilde T}^*$ by (i). Finally because $P_{\m K_2}\, \xi = \xi_{\m K_2}$ we obtain
$P_{\m K_2} \; \ker T^* \subset \ker {\tilde T}^*$
which is (ii).

\noindent(iii)  
First note that if 
$\xi \in \ker T^* \cap\ \m K_2$ then
$\xi \perp \overline {{\rm range}\, T} \supset  T \m N_1$
and $\xi \perp  \m N_2$. Therefore  $\ker T^* \cap\ \m K_2 \perp \m L$.

Further $\m L \subset \m N_2^\perp = \m K_2$ and $T^* \m L \subset T^*\big[\overline{\rm span}\,\{\m N_2, T  \m N_1\} \big] = \m N_1$, by the assumption that $T$ is an isometry. So $\m L \subset \ker {\tilde T}^*$ by (i). Clearly
$\ker T^* \cap\ \m K_2 \subset \ker {\tilde T}^*$ by (i). Together we have $\m L \oplus(\ker  T^*\cap\ \m K_2) \subset \ker {\tilde T}^*$.

For the opposite inclusion let $\xi \in \ker {\tilde T}^* \ominus \m L$.
Because $\xi \in \ker {\tilde T}^*$ we have $\xi \in \m K_2$ and
$T^* \xi \in \m N_1$ by (i). But from $\xi \perp \m L$ and $\xi \perp \m N_2$ we also get, using the definition of $\m L$, that $\xi \perp T \m N_1, i.e., T^* \xi \perp \m N_1$. Hence $T^* \xi = 0$ and
$\xi \in \ker  T^*\cap\ \m K_2$. We have now established the second equality in (iii).

Finally it is clear that $\ker T^* \cap \m K_2 \subset P_{\m K_2}\,\ker T^*$, hence $\m L \oplus(\ker  T^*\cap\ \m K_2) \subset
\overline{\rm span}\,\{\m L, P_{\m K_2}\, \ker  T^*\}$. On the other hand we have seen above that $\m L \subset \ker {\tilde T}^*$ and we have $P_{\m K_2}\,\ker T^* \subset \ker {\tilde T}^*$ by (ii).
Hence we also have 
$\overline{\rm span}\,\{\m L, P_{\m K_2}\, \ker  T^*\}
\subset \ker {\tilde T}^* = \m L \oplus(\ker  T^*\cap\ \m K_2)$ and we have proved the first equality in (iii).
\end{proof}

\vspace{0.5cm}

{\bf Acknowledgement:} We would like to thank both referees for extraordinarily rich comments resulting in a number of improvements of the original manuscript.

\end{document}